\documentclass[reqno]{amsart}

\newcommand{\Rm}{\mathbb{R}}
\newcommand{\Cm}{\mathbb{C}}

\newcommand{\Nm}{\mathbb{N}}

\newcommand{\ba}{\begin{eqnarray*}}
\newcommand{\ea}{\end{eqnarray*}}
\newcommand{\be}{\begin{equation}}
\newcommand{\ee}{\end{equation}}
\newcommand{\bea}{\begin{eqnarray}}
\newcommand{\eea}{\end{eqnarray}}
\newcommand{\va}{\varphi}
\newcommand{\pp}{\partial}
\newcommand{\vv}[1]{\boldsymbol{\mathrm{#1}}}
\newcommand{\hvv}[1]{\boldsymbol{\hat{\mathrm{#1}}}}

\usepackage{amsmath}
\usepackage[foot]{amsaddr}
\usepackage{amsthm,amssymb,mathrsfs}

\theoremstyle{plain}
\newtheorem{thm}{Theorem}[section]
\newtheorem{lem}[thm]{Lemma}

\newtheorem{prop}[thm]{Proposition}
\theoremstyle{definition}
\theoremstyle{remark}\newtheorem{rmk}{Remark}

\title[]{Global Lipschitz stability for a fractional inverse transport problem by Carleman estimates}

\author{Atsushi Kawamoto}
\address{Department of Mathematical Sciences, The University of Tokyo, Komaba, Meguro, Tokyo 153-8914, Japan}
\email{kawamo@ms.u-tokyo.ac.jp}
\author{Manabu Machida}
\address{Institute for Medical Photonics Research, Hamamatsu University School of Medicine, Hamamatsu, Shizuoka 431-3192, Japan}
\email{machida@hama-med.ac.jp}

\date{{\small May 8, 2018}}

\begin{document}

\begin{abstract}
We consider a fractional radiative transport equation, where the time 
derivative is of half order in the Caputo sense. By establishing Carleman 
estimates, we prove the global Lipschitz stability in determining the 
coefficients of the one-dimensional time-fractional radiative transport 
equation of half-order.
\end{abstract}

\maketitle

\section{Introduction}

Let us consider the following time fractional radiative transport equation 
with the initial condition and Cauchy data in one dimension.
\be
\left\{\begin{aligned}
\left(\pp_t^{1/2}+v\pp_x+\sigma_t(x,v)\right)u(x,v,t)
&
\\
\qquad=\sigma_s(x,v)\int_Vp(x,v,v')u(x,v',t)\,dv',
&\quad(x,t)\in Q,\quad v\in V,
\\
u(x,v,0)=a(x,v),
&\quad x\in\Omega,\quad v\in V,
\\
u(x,v,t)=g(x,v,t),
&\quad (x,v)\in\Gamma_-,\quad t\in(0,T),
\end{aligned}\right.
\label{rte1}
\ee
where $\pp_t^{1/2}$ is the Caputo fractional derivative \cite{Caputo67} 
of half order given by
\[
\pp_t^{1/2}u(\cdot,\cdot,t)=\frac{1}{\Gamma\left(\frac{1}{2}\right)}
\int_0^t\frac{\pp_{\tau}u(\cdot,\cdot,\tau)}{\sqrt{t-\tau}}\,d\tau.
\]
We note that $\Gamma(\cdot)$ is the gamma function and 
$\Gamma\left(\frac{1}{2}\right)=\sqrt{\pi}$. Here we defined
\[
Q=\{(x,t);\;x\in\Omega,\,0<t<T\},\qquad\Omega=(0,\ell),\qquad
V=\{v\in\mathbb{R};\;v_0\le|v|\le v_1\},
\]
with positive constants $\ell,v_0,v_1$. We define $\Gamma_+$ and 
$\Gamma_-$ by
\[
\Gamma_{\pm}=\{(x,v)\in\pp\Omega\times V;\;
\pm v<0\;\mbox{at}\;x=0,\;\pm v>0\;\mbox{at}\;x=\ell\}.
\]
That is, for a function $f(x,v)$, we have
\ba
\int_{\Gamma_+}f(x,v)\,dSdv&=&
\int_{-v_1}^{-v_0}f(0,v)\,dv+\int_{v_0}^{v_1}f(\ell,v)\,dv,
\\
\int_{\Gamma_-}f(x,v)\,dSdv&=&
\int_{v_0}^{v_1}f(0,v)\,dv+\int_{-v_1}^{-v_0}f(\ell,v)\,dv.
\ea
We assume
\[
\sigma_t\in C^1(\overline{\Omega};L^{\infty}(V)),\qquad
\sigma_s\in C^1(\overline{\Omega};L^{\infty}(V)),
\]
and
\[
p\in C^1(\overline{\Omega};L^{\infty}(V\times V)),\qquad
p>0\quad\mbox{in}\;\Omega\times V\times V.
\]
The phase function $p(x,v,v')$ is assumed to be known, whereas $\sigma_t$, 
$\sigma_s$, or both are unknown.

\begin{rmk}
Anomalous diffusion is said to be subdiffusion when $\alpha\in(0,1)$. In the case of the time-fractional diffusion equation, analysis for $\alpha=n/m$ ($m,n\in\Nm$, $m>n$) is possible once we establish the methodology for $\alpha=1/2$ \cite{Xu-Cheng-Yamamoto11}. Similarly, we can in principle use the general $\alpha$ after we develop in the present paper the analysis for the time-fractional radiative transport equation for $\alpha=1/2$.
\end{rmk}

The time-fractional radiative transport equation is approximated by the time-fractional diffusion equation in the asymptotic limit \cite{Machida17}. Inverse problems for time-fractional diffusion equations with the Caputo derivative $\pp_t^{\alpha}$ have been intensively studied during the last decade. Uniqueness in determining $\alpha$ and the diffusion coefficient was proven \cite{Cheng-etal09}. A Carleman estimate was established for the time-fractional diffusion equation with $\alpha=1/2$ \cite{Xu-Cheng-Yamamoto11}. Using the Carleman estimate technique, conditional stability in determining a zeroth-order coefficient for $\alpha=1/2$ was proven \cite{Yamamoto-Zhang12}. Recovering the absorption coefficient was considered \cite{Jin-Rundell12}. A reconstruction scheme for $\alpha$ was given in \cite{Hatano-etal13}. Simultaneous reconstruction of the initial status and boundary value was considered \cite{Liu-Yamamoto-Yan15}. Recently, unique continuation was proved for arbitrary $\alpha$ \cite{Lin-Nakamura16}.

In this paper, we prove the global Lipschitz stability when determining $\sigma_t(x,v)$ and $\sigma_s(x,v)$ from boundary measurements. The proof is based on Carleman estimates first established in \cite{Carleman39}. The methodology was first used in inverse problems for proving the global uniqueness \cite{Bukhgeim-Klibanov81}. See \cite{Klibanov13} and references therein. Our proof particularly relies on the method developed to show the global Lipschitz stability for the inverse source problem of parabolic equations \cite{Imanuvilov-Yamamoto98}. See a review article \cite{Yamamoto09} for further details. For the usual radiative transport equation with $\pp_t$, the Lipschitz stability was shown for $-T<t<T$ \cite{Klibanov-Pamyatnykh06}, for the purely absorbing case of $\sigma_s\equiv0$ \cite{Gaitan-Ouzzane13}, and for $0<t<T$ \cite{Machida-Yamamoto14}. The recovery of $\sigma_t$ was also considered in \cite{Acosta15}. The exact controllability was proved \cite{Klibanov-Yamamoto07} and the case that $\sigma_t$ depends on $x,v,t$ was considered in \cite{Prilepko-Ivankov84}. See \cite{Bal09} and references therein for the H\"{o}lder-type stability analysis using the albedo operator.

The remainder of this paper is organized as follows. In \S\ref{main}, main results are stated. We give some physical background in \S\ref{atran}. In \S4, we derive a first-order equation in time by multiplying $\pp_t^{1/2}$ by the fractional radiative transport equation in (\ref{rte1}). In \S5, we establish our key Carleman estimate. In \S6, we prove Theorems \ref{thm:bd}, \ref{thm:bdt}, and \ref{thm:bds}. Another Carleman estimate necessary in \S6 is derived in Appendix A.

\section{Main results}
\label{main}

We define
\[
X=H^2(0,T;H^{1,0}(\Omega\times V))\cap L^{\infty}(0,T;H^{2,0}(\Omega\times V)).
\]
For an arbitrarily fixed constant $M>0$, we set
\[
\mathcal{U}=\left\{u\in X;\;
\|u\|_X+\|\pp_xu\|_{H^1(\Omega\times(0,T);L^2(V))}\le M\right\}.
\]
Let $t_0$ be an arbitrarily fixed time on $(0,T)$. We take $\delta>0$ 
such that 
\[
0<t_0-\delta<t_0 < t_0+\delta <T. 
\]
Moreover we set
\[
Q_\delta=\Omega\times (t_0-\delta,t_0+\delta)\quad\mbox{for}\;
0<\delta<\min(t_0,T-t_0).
\]

Let us consider two total attenuations $\sigma_t^{(1)}(x,v)$ and 
$\sigma_t^{(2)}(x,v)$ with $\sigma_t^{(1)}(0,v)=\sigma_t^{(2)}(0,v)$ for all 
$v\in V$, and two scattering coefficients $\sigma_s^{(1)}(x,v)$ and 
$\sigma_s^{(2)}(x,v)$ with $\sigma_s^{(1)}(0,v)=\sigma_s^{(2)}(0,v)$ for all 
$v\in V$. We perform boundary measurements twice for the pairs of initial and 
boundary values, $(a_1,g_1)$ and $(a_2,g_2)$. Let $u_j^{(1)}$ and $u_j^{(2)}$ 
be the corresponding solutions to (\ref{rte1}) for $a_j(x,v)$ and 
$g_j(x,v,t)$ ($j=1,2$). We introduce a $2\times2$ matrix $R(x,v,t)$ as
\[
R(x,v,t)=\left(\begin{array}{cc}
-u_1^{(2)}(x,v,t) & \quad\int_Vp(x,v,v')u_1^{(2)}(x,v',t)\,dv' \\
-u_2^{(2)}(x,v,t) & \quad\int_Vp(x,v,v')u_2^{(2)}(x,v',t)\,dv'
\end{array}\right).
\]
We choose $(a_1,g_1)$ and $(a_2,g_2)$ so that $\det{R(x,v,t_0)}\neq0$ is 
satisfied for a chosen time $t_0\in(0,T)$.

\begin{thm}[Simultaneous determination of $\sigma_t$, $\sigma_s$]
\label{thm:bd}
Let $u_j^{(i)}\in\mathcal{U}$ ($i=1,2$, $j=1,2$), 
$\|\sigma_t^{(i)}\|_{L^{\infty}(\Omega\times V)}\le M$ ($i=1,2$), and 
$\|\sigma_s^{(i)}\|_{L^{\infty}(\Omega\times V)}\le M$ ($i=1,2$). Moreover, we 
suppose $u_j^{(2)}\in C^1(\overline{Q_\delta};L^\infty(V))$, 
$\pp_t^{1/2}u_j^{(2)}\in C^1([t_0-\delta,t_0+\delta];L^\infty(\Omega\times V))$ for $j=1,2$.
We assume that $\det{R(\cdot,\cdot,t_0)}\neq0$ in $\overline{\Omega\times V}$. 
Then there exists $C=C(t_0,\delta,M)>0$ such that
\ba
&&
\|\sigma_t^{(1)}-\sigma_t^{(2)}\|_{H^1(\Omega;L^2(V))}^2+
\|\sigma_s^{(1)}-\sigma_s^{(2)}\|_{H^1(\Omega;L^2(V))}^2
\\
&&\le
C\sum_{j=1}^2\left\|u_j^{(1)}(\cdot,\cdot,t_0)-u_j^{(2)}(\cdot,\cdot,t_0)
\right\|_{H^2(\Omega;L^2(V))}^2
\\
&&+
C\sum_{j=1}^2\int_{t_0-\delta}^{t_0+\delta}\int_{\Gamma_+}
\left[\left|\pp_t(u_j^{(1)}-u_j^{(2)})\right|^2+
\left|\pp_t^2(u_j^{(1)}-u_j^{(2)})\right|^2+
\left|\pp_t\pp_x(u_j^{(1)}-u_j^{(2)})\right|^2\right]\,dSdvdt
\\
&&+
C\sum_{j=1}^2\int_{t_0-\delta}^{t_0+\delta}\int_V
\left|\pp_x\pp_t\left(u_j^{(1)}(0,v,t)-u_j^{(2)}(0,v,t)\right)\right|^2\,dvdt,
\ea
where $0<\delta<\min(t_0,T-t_0)$. Here, $C(t_0,\delta,M)\to\infty$ 
as $M\to\infty$.
\end{thm}

If one of the coefficients is known, we can determine $\sigma_t$ or 
$\sigma_s$ from a single measurement. The following theorems can be proved 
similar to Theorem \ref{thm:bd}.

\begin{thm}[Determination of $\sigma_t$]
\label{thm:bdt}
Let $u^{(i)}\in\mathcal{U}$ ($i=1,2$), 
$\|\sigma_t^{(i)}\|_{L^{\infty}(\Omega\times V)}\le M$ ($i=1,2$). Moreover we 
suppose $u^{(2)}\in C^1(\overline{Q_\delta};L^\infty(V))$, 
$\pp_t^{1/2}u^{(2)}\in C^1([t_0-\delta,t_0+\delta];L^\infty(\Omega\times V))$, and 
$u^{(2)}(\cdot,\cdot,t_0)\neq 0$ in $\overline{\Omega\times V}$. Then 
there exists $C=C(t_0,\delta,M)>0$ such that
\ba
&&
\|\sigma_t^{(1)}-\sigma_t^{(2)}\|_{H^1(\Omega;L^2(V))}^2
\\
&&\le
C\left\|u^{(1)}(\cdot,\cdot,t_0)-u^{(2)}(\cdot,\cdot,t_0)
\right\|_{H^2(\Omega;L^2(V))}^2
\\
&&+
C\int_{t_0-\delta}^{t_0+\delta}\int_{\Gamma_+}
\left[\left|\pp_t(u^{(1)}-u^{(2)})\right|^2+
\left|\pp_t^2(u^{(1)}-u^{(2)})\right|^2+
\left|\pp_t\pp_x(u^{(1)}-u^{(2)})\right|^2\right]\,dSdvdt
\\
&&+
C\int_{t_0-\delta}^{t_0+\delta}\int_V
\left|\pp_x\pp_t\left(u^{(1)}(0,v,t)-u^{(2)}(0,v,t)\right)\right|^2\,dvdt,
\ea
where $0<\delta<\min(t_0,T-t_0)$. Here, $C(t_0,\delta,M)\to\infty$ 
as $M\to\infty$.
\end{thm}

\begin{thm}[Determination of $\sigma_s$]
\label{thm:bds}
Let $u^{(i)}\in\mathcal{U}$ ($i=1,2$), 
$\|\sigma_s^{(i)}\|_{L^{\infty}(\Omega\times V)}\le M$ ($i=1,2$). Moreover we 
suppose $u^{(2)}\in C^1(\overline{Q_\delta};L^\infty(V))$, 
$\pp_t^{1/2}u^{(2)}\in C^1([t_0-\delta,t_0+\delta];L^\infty(\Omega\times V))$, and 
$\int_V p(\cdot, \cdot , v')u^{(2)}(\cdot,v',t_0)\,dv'\neq 0$ in 
$\overline{\Omega\times V}$. Then there exists $C=C(t_0,\delta,M)>0$ such that
\ba
&&
\|\sigma_s^{(1)}-\sigma_s^{(2)}\|_{H^1(\Omega;L^2(V))}^2
\\
&&\le
C\left\|u^{(1)}(\cdot,\cdot,t_0)-u^{(2)}(\cdot,\cdot,t_0)
\right\|_{H^2(\Omega;L^2(V))}^2
\\
&&+
C\int_{t_0-\delta}^{t_0+\delta}\int_{\Gamma_+}
\left[\left|\pp_t(u^{(1)}-u^{(2)})\right|^2+
\left|\pp_t^2(u^{(1)}-u^{(2)})\right|^2+
\left|\pp_t\pp_x(u^{(1)}-u^{(2)})\right|^2\right]\,dSdvdt
\\
&&+
C\int_{t_0-\delta}^{t_0+\delta}\int_V
\left|\pp_x\pp_t\left(u^{(1)}(0,v,t)-u^{(2)}(0,v,t)\right)\right|^2\,dvdt,
\ea
where $0<\delta<\min(t_0,T-t_0)$. Here, $C(t_0,\delta,M)\to\infty$ 
as $M\to\infty$.
\end{thm}

\begin{rmk}
\label{apriori}
In Theorem \ref{thm:bd}, we need an a priori assumption $\det R(\cdot,\cdot, t_0)\neq 0$ in $\overline{\Omega\times V}$ at the observation time $t=t_0$. This nonzero condition is satisfied by the appropriate choice of $(a_j, g_j)$ for $j=1,2$. The controllability result for \eqref{rte1} about how to choose $(a_j, g_j)$ for $j=1,2$ is not yet known but obtained along the same lines of the calculation (in particular, Proposition 1.1) by Yuan and Yamamoto \cite{Yuan-Yamamoto09}, which is concerned with a parabolic equation. See also Remark 1.3 in Machida and Yamamoto \cite{Machida-Yamamoto14} for the radiative transport equation.
\end{rmk}

\section{Anomalous transport}
\label{atran}

\subsection{Relation to anomalous diffusion and anomalous transport}

Anomalous diffusion is often studied using fractional diffusion equations \cite{Metzler-Klafter00,Sokolov-etal00}. In particular, anomalous diffusion is observed for tracer particles moving in an aquifer \cite{Adams-Gelhar92}. An analysis of column experiments revealed a power-law behavior of the waiting-time function of the continuous-time random walk \cite{Hatano-Hatano98}, which has motivated the use of the fractional diffusion equations. However, recent study shows that such fractional diffusion equations fail to explain the flow of tracer particles in column experiments especially during short time periods \cite{Yamakawa12}. When considering the fact that the time-fractional diffusion equation is obtained in the asymptotic limit of the time-fractional radiative transport equation for long time and large distance \cite{Machida17}, our attention is driven to the study of the latter equation as a more accurate model of anomalous transport.

It is known that the mass distribution of tracer particles moving in an aquifer reveals non-Gaussian behavior \cite{Adams-Gelhar92} and the linear Boltzmann transport has been proposed \cite{Williams92,Williams93}. Theorem \ref{thm:bd} guarantees the global Lipschitz stability in determining the absorption and scattering properties of the area of interest when the concentration of tracer particles is measured with pumping wells surrounding the area. Also Theorem \ref{thm:bd} might be related to optical tomography \cite{Arridge-Schotland09}, in which optical properties of absorption and scattering are determined from boundary measurements, if propagation of light for some reason shows anomalous transport.

\subsection{Continuous-time random walk}
\label{ctrw}

The fractional diffusion equation is derived from the continuous-time random walk. In the same manner, the fractional radiative transport equation is derived from the continuous-time random walk with velocity.

We begin with the usual continuous-time random walk in $x\in\Rm$, $t\ge0$. Let $\va(x,t)$ be the jump probability density function given by $\va(x,t)=\lambda(x)w(t)$, where $\lambda(x)$ is the jump length probability density function and $w(t)$ is the waiting time probability density function. They are calculated as $\lambda(x)=\int_0^{\infty}\va(x,t)\,dt$, $w(t)=\int_{-\infty}^{\infty}\va(x,t)\,dx$. Using $\va(x,t)$, the probability density function $\eta(x,t)$ of just having arrived at position $x$ at time $t$ is written as
\[
\eta(x,t)=
\int_0^t\int_{-\infty}^{\infty}\eta(y,s)\va(x-y,t-s)\,dyds+a(x)\delta(t),
\]
where $a(x)$ is the initial value. We note that the cumulative probability $\Phi(t)$ of not having moved during $t$ is given by
\be
\Phi(t)=1-\int_0^tw(s)\,ds.
\label{ctrw:Phi}
\ee
Thus the probability density function $P(x,t)$ of being at $(x,t)\in\Rm\times[0,\infty)$ is obtained as $P(x,t)=\int_0^t\eta(x,s)\Phi(t-s)\,ds$. Suppose that the Fourier transform of $\lambda(x)$ behaves like $(\mathcal{F}\lambda)(k)\sim1-\sigma^2k^2$ for small $k$ and the Laplace transform of $w(t)$ behaves like $(\mathcal{L}w)(s)\sim1-(\tau s)^{\alpha}$ for small $s$ ($0<\alpha\le1$). Then it is known that $P(x,t)$ asymptotically obeys the following diffusion equation ($\alpha=1$) or time-fractional diffusion equation ($0<\alpha<1$) in the limit of large $x$ and large $t$ (see, for example, \cite{Metzler-Klafter00}).
\[
\pp_t^{\alpha}P-\frac{\sigma^2}{\tau^{\alpha}}\pp_x^2P=0.
\]

Now, we generalize $\lambda$ taking velocity into account \cite{Machida17}. Absorption is also considered. We give $\lambda(x;v,v')$ as
\[
\lambda(x;v,v')=
\xi_s\delta(x)p(v,v')+\left(1-\xi_t\right)\delta(x-v\tau_0)\delta(v-v'),
\]
where $\xi_t\in(0,1)$, $\xi_s\in(0,\xi_t)$, and $\tau_0>0$ are constants. We will give $\tau_0$ below depending on $w(t),\xi_t$. The first term on the right-hand side is the probability that there is no jump but the velocity changes from $v'$ to $v$. The function $p(v,v')$ is the probability that the target particle changes its velocity from $v'$ to $v$ when it is scattered by a scatterer. The second term shows the probability of transport that the target particle jumps keeping its velocity without being scattered nor absorbed. Correspondingly we give $\va(x,t;v,v')$ as $\va(x,t;v,v')=\lambda(x;v,v')w(t)$, with the relations $\lambda(x;v,v')=\int_0^{\infty}\va(x,t;v,v')\,dt$, $\left(1-\xi_a\right)w(t)=
\int_V\int_{-\infty}^{\infty}\va(x,t;v,v')\,dxdv'$, where we introduced the probability $\xi_a=\xi_t-\xi_s>0$ for absorption. Then we have
\[
\eta(x,v,t)=\int_0^t\int_V\int_{-\infty}^{\infty}\eta(y,v',s)
\va(x-y,t-s;v,v')\,dydv'ds+a(x,v)\delta(t).
\]
With this $\eta(x,v,t)$, the probability density function $P(x,v,t)$ of being at $(x,v,t)\in\Rm\times V\times[0,\infty)$ is given by $P(x,v,t)=\int_0^t\eta(x,v,t)\Phi(t-s)\,ds$, where $\Phi$ is introduced in (\ref{ctrw:Phi}). In the asymptotic limit of small $k,s$, we obtain
\[
\left(\pp_t^{\alpha}+v\pp_x+\sigma_t\right)P(x,v,t)=
\sigma_s\int_Vp(v,v')P(x,v',t)\,dv',
\]
where $\sigma_t=\xi_t/\tau^{\alpha}$, $\sigma_s=\xi_s/\tau^{\alpha}$, $\tau_0=\tau^{\alpha}/(1-\xi_t)$. Thus we see that (\ref{rte1}) is related to the continuous-time random walk with velocity. Furthermore it can be shown that (\ref{rte1}) reduces to the diffusion equation with the absorption term in the asymptotic limit \cite{Machida17}. In this sense, (\ref{rte1}) governs anomalous transport at the mesoscopic scale, whereas the governing equation is the fractional diffusion equation at the macroscopic scale.

\section{From one-half to one}
\label{half2one}

Since we have no Carleman estimates for time-fractional radiative transport 
equations, we begin by obtaining an equation with the time derivative 
$\pp_t$ by taking the $t$-derivative of half-order in the original equation. 
The following lemma ensures the relation 
$\pp_t^{1/2}\pp_t^{1/2}=\pp_t$ in the calculation developed in this section.

\begin{lem}[Xu-Cheng-Yamamoto \cite{Xu-Cheng-Yamamoto11}]
Let $\tilde{u}\in C[0,T]\cap W^{1,1}(0,T)$ and
\[
\tilde{u}(0)=\pp_t^{\alpha_2}\tilde{u}(0)=0.
\]
Then for $0<\alpha_1+\alpha_2\le 1$,
\[
\pp_t^{\alpha_1}\pp_t^{\alpha_2}\tilde{u}(t)=
\pp_t^{\alpha_1+\alpha_2}\tilde{u}(t).
\]
\end{lem}

Let us consider differences,
\[
r_t(x,v)=\sigma_t^{(1)}(x,v)-\sigma_t^{(2)}(x,v),\qquad
r_s(x,v)=\sigma_s^{(1)}(x,v)-\sigma_s^{(2)}(x,v),
\]
where $r_t(x,v),r_s(x,v)\in C^1(\Omega;L^{\infty}(V))$ with $r_t(0,v)=r_s(0,v)=0$ for $v\in V$. We define vectors $\vv{r},\vv{u}$ as
\[
\vv{r}(x,v)=\left(\begin{array}{c}
r_t(x,v) \\ r_s(x,v)
\end{array}\right),\qquad
\vv{u}(x,v,t)=\left(\begin{array}{c}
u_1^{(1)}(x,v,t)-u_1^{(2)}(x,v,t) \\ u_2^{(1)}(x,v,t)-u_2^{(2)}(x,v,t)
\end{array}\right).
\]
Similar to Yamamoto and Zhang \cite{Yamamoto-Zhang12}, we introduce 
a new vector $\hvv{u}(x,v,t)$ as
\[
\hvv{u}(x,v,t)=\vv{u}(x,v,t)-
\frac{2t^{1/2}}{\Gamma\left(\frac{1}{2}\right)}R(x,v,0)\vv{r}(x,v).
\]
By differentiating both sides 
of the above equation with respect to $t$, we obtain $\pp_t\hvv{u}$ as
\be
\pp_t\hvv{u}(x,v,t)=\pp_t\vv{u}(x,v,t)-
\frac{1}{\Gamma\left(\frac{1}{2}\right)t^{1/2}}
R(x,v,0)\vv{r}(x,v).
\label{ppthatu1}
\ee
We note that
\[
\hvv{u}(x,v,0)=0.
\]
Using $\pp_t^{1/2}t^{1/2}=\frac{1}{2}\Gamma(1/2)$, we obtain
\[
\pp_t^{1/2}\hvv{u}(x,v,t)=\pp_t^{1/2}\vv{u}(x,v,t)-R(x,v,0)\vv{r}(x,v).
\]
The above equation implies
\[
\pp_t^{1/2}\hvv{u}(x,v,0)=0.
\]
We note that by writing $\sigma_t^{(1)}$, $\sigma_s^{(1)}$ as $\sigma_t$, $\sigma_s$, we obtain the following time-fractional radiative transport equation.
\be
\left\{\begin{aligned}
\left(\pp_t^{1/2}+v\pp_x+\sigma_t(x,v)\right)\vv{u}(x,v,t)
=\sigma_s(x,v)\int_Vp(x,v,v')\vv{u}(x,v',t)\,dv'
\\
\qquad+R(x,v,t)\vv{r}(x,v),\quad (x,t)\in Q,\quad v\in V,
\\
\vv{u}(x,v,0)=\vv{0},
\quad x\in\Omega,\quad v\in V,
\\
\vv{u}(x,v,t)=\vv{0},
\quad(x,v)\in\Gamma_-,\quad t\in(0,T).
\end{aligned}\right.
\label{rte2}
\ee
Now we can alternatively compute $\pp_t\hvv{u}$ as follows.
\bea
&&
\pp_t\hvv{u}
=\pp_t^{1/2}\pp_t^{1/2}\hvv{u}
\nonumber \\
&&=
\pp_t^{1/2}\left(-v\pp_x\vv{u}-\sigma_t\vv{u}+\sigma_s\int_Vp\vv{u}\,dv'+
R\vv{r}\right)
\nonumber \\
&&=
-v\pp_x\left(-v\pp_x\vv{u}-\sigma_t\vv{u}+\sigma_s\int_Vp\vv{u}\,dv'+
R\vv{r}\right)
\nonumber \\
&&-
\sigma_t\left(-v\pp_x\vv{u}-\sigma_t\vv{u}+\sigma_s\int_Vp\vv{u}\,dv'+
R\vv{r}\right)
\nonumber \\
&&+
\sigma_s\int_Vp\Bigl(-v'\pp_x\vv{u}(x,v',t)-\sigma_t(x,v')\vv{u}(x,v',t)
\nonumber \\
&&+
\sigma_s(x,v')\int_Vp\vv{u}\,dv''+R(x,v',t)\vv{r}(x,v')\Bigr)\,dv'
+\left(\pp_t^{1/2}R\right)\vv{r}
\nonumber \\
&&=
v^2\pp_x^2\vv{u}-vR\pp_x\vv{r}-\left(v\pp_xR+\sigma_t R-\pp_t^{1/2}R\right)\vv{r}
\nonumber \\
&&+
2v\sigma_t\pp_x\vv{u}+\left(v(\pp_x\sigma_t)+\sigma_t^2\right)\vv{u}
-\left(v\pp_x+\sigma_t\right)\sigma_s\int_Vp(x,v,v')\vv{u}(x,v',t)\,dv'
\nonumber \\
&&+
\sigma_s\int_Vp\Bigl(-v'\pp_x\vv{u}(x,v',t)-\sigma_t(x,v')\vv{u}(x,v',t)
\nonumber \\
&&+
\sigma_s\int_Vp\vv{u}\,dv''+R(x,v',t)\vv{r}(x,v')\Bigr)\,dv'.
\label{ppthatu2}
\eea
From (\ref{ppthatu1}) and (\ref{ppthatu2}), we arrive at the following 
equation.
\be
\left\{\begin{aligned}
\pp_t\vv{u}(x,v,t)-v^2\pp_x^2\vv{u}-L_1\vv{u}(x,v,t)=
\int_VK(x,v,v')\vv{u}(x,v',t)\,dv'
\\
\qquad+\vv{f}(x,v,t),
\quad (x,t)\in Q,\quad v\in V,
\\
\vv{u}(x,v,0)=\vv{0},
\quad x\in\Omega,\quad v\in V,
\\
\vv{u}(x,v,t)=\vv{0},
\quad(x,v)\in\Gamma_-,\quad t\in(0,T).
\end{aligned}\right.
\label{rte3}
\ee
Here,
\ba
L_1\vv{u}(x,v,t)
&=&
2v\sigma_t(x,v)\pp_x\vv{u}(x,v,t)+
\left(v\pp_x\sigma_t(x,v)+\sigma_t^2(x,v)\right)\vv{u}(x,v,t),
\\
K(x,v,v')
&=&
-v\pp_x\left(\sigma_s(x,v)p(x,v,v')\right)
\\
&-&
\sigma_s(x,v)p(x,v,v')\left((v+v')\pp_x+\sigma_t(x,v)+\sigma_t(x,v')\right)
\\
&+&
\sigma_s(x,v)\int_V\sigma_s(x,v'')p(x,v,v'')p(x,v'',v')\,dv'',
\ea
and
\bea
&&
\vv{f}(x,v,t)
=
-vR(x,v,t)\pp_x\vv{r}(x,v)
\nonumber \\
&&-
\left[v\pp_xR(x,v,t)+\sigma_t(x,v)R(x,v,t)-\pp_t^{1/2}R(x,v,t)-
\frac{1}{\Gamma\left(\frac{1}{2}\right)t^{1/2}}R(x,v,0)\right]\vv{r}(x,v)
\nonumber \\
&&+
\sigma_s(x,v)\int_Vp(x,v,v')R(x,v',t)\vv{r}(x,v')\,dv'.
\label{defoff}
\eea

\begin{rmk}
Our argument only works in one dimension. In the multi-dimensional case ($n>1$), the principal coefficients $v_iv_j$ ($i,j=1,\dots,n$) of the parabolic equation corresponding to \eqref{rte3} do not satisfy the uniform ellipticity for $V=\{v\in\mathbb{R}^n;\;v_0\le|v|\le v_1\}$, and we can not derive the Carleman estimate, which is obtained below for the one-dimensional equation. 
\end{rmk}

\section{Carleman estimate}

Hereafter in this paper, we let $C$ denote generic positive constants. 
Let us look at one component of the vector equation (\ref{rte3}) and 
consider the following equation.
\be
\left\{\begin{aligned}
L_0u(x,v,t)-L_1u(x,v,t)-\int_VK(x,v,v')u(x,v',t)\,dv'
=f(x,v,t),
\\
\qquad(x,t)\in Q,\quad v\in V,
\\
u(x,v,0)=0,
\quad x\in\Omega,\quad v\in V,
\\
u(x,v,t)=0,
\quad(x,v)\in\Gamma_-,\quad t\in(0,T),
\end{aligned}\right.
\label{rte4}
\ee
where
\[
L_0u(x,v,t)=\pp_tu(x,v,t)-v^2\pp_x^2u(x,v,t).
\]

Let $d\in C^2(\overline{\Omega})$ be a function such that
\[
d(x)>0\quad\mbox{for}\quad x\in\Omega,\qquad
\pp_x d(x)<0\quad\mbox{for}\quad x\in\overline{\Omega}.
\]

As was done in \cite{Emanuilov95,Fursikov-Imanuvilov96,Imanuvilov-Yamamoto98,Yamamoto09}, we use the weight function $\alpha$ as
\be
\alpha(x,t)=
\frac{e^{\lambda d(x)}-e^{2\lambda\|d\|_{C(\overline{\Omega})}}}{t(T-t)}.
\label{weightalpha}
\ee
We define
\[
\va(x,t)=\frac{e^{\lambda d(x)}}{t(T-t)}.
\]
We set
\[
z(x,v,t)=e^{s\alpha(x,t)}u(x,v,t).
\]
We note that $\alpha<0$ in $\Omega\times(0,T)$, and
\[
z(x,v,0)=z(x,v,T)=0,\qquad\pp_xz(x,v,0)=\pp_xz(x,v,T)=0,
\]
for $(x,v)\in\Omega\times V$.

\begin{prop}[Carleman estimate]
\label{prop:bd}
There exists $\lambda_0>0$ such that for arbitrary $\lambda>\lambda_0$, we can choose $s_0=s_0(\lambda)>0$ and there exists $C=C(s_0,\lambda_0)>0$ such that the following estimate holds for all $s>s_0$ and all $u\in\mathcal{U}$ which satisfies (\ref{rte4}). 
\bea
&&
\int_{Q\times V}\left[\frac{1}{s\va}|\pp_tu|^2+
s\lambda^2\va|\pp_xu|^2+s^3\lambda^4\va^3|u|^2\right]e^{2s\alpha}\,dxdvdt
\nonumber \\
&&\le
C\int_{Q\times V}|f|^2e^{2s\alpha}\,dxdvdt
+Ce^{C(\lambda)s}\int_0^T\int_{\Gamma_+}\left(|u|^2+|\pp_tu|^2+|\pp_xu|^2
\right)\,dSdvdt
\nonumber \\
&&+
Ce^{C(\lambda)s}\int_V\int_0^T|\pp_x u(0,v,t)|^2\,dtdv.
\label{prop:bd:eq1}
\eea
\end{prop}

\begin{proof}
It is sufficient to show the Carleman estimate for $L_0u$. Suppose we have
\bea
&&
\int_V\int_Q\left[\frac{1}{s\va}|\pp_tu|^2+
s\lambda^2\va|\pp_xu|^2+s^3\lambda^4\va^3|u|^2\right]e^{2s\alpha}\,dxdtdv
\nonumber \\
&&\le
C\int_V\int_Q|L_0u|^2e^{2s\alpha}\,dxdtdv
+Ce^{C(\lambda)s}\int_0^T\int_{\Gamma_+}\left(|u|^2+|\pp_tu|^2+|\pp_xu|^2
\right)\,dSdvdt
\nonumber \\
&&
+Ce^{C(\lambda)s}\int_V\int_0^T|\pp_x u(0,v,t)|^2\,dtdv.
\label{principal}
\eea
Since
\ba
|L_0u|^2
&\le&
C|f|^2+C|L_1u|^2+C\left|\int_VK(x,v,v')u(x,v',t)\,dv'\right|^2
\\
&\le&
C|f|^2+C|\pp_xu|^2+C|u|^2+C\left|\int_VK(x,v,v')u(x,v',t)\,dv'\right|^2,
\ea
we obtain
\ba
&&
\int_V\int_Q\left[\frac{1}{s\va}|\pp_tu|^2+
s\lambda^2\va|\pp_xu|^2+s^3\lambda^4\va^3|u|^2\right]e^{2s\alpha}\,dxdtdv
\\
&&\le
\int_V\int_Q\left[C|f|^2+C|\pp_xu|^2+C|u|^2+
C\left|\int_VK(x,v,v')u(x,v',t)\,dv'\right|^2\right]e^{2s\alpha}\,dxdtdv
\\
&&+
Ce^{C(\lambda)s}\int_0^T\int_{\Gamma_+}\left(|u|^2+|\pp_tu|^2+|\pp_xu|^2
\right)\,dSdvdt 
\\
&&
+Ce^{C(\lambda)s}\int_V\int_0^T|\pp_x u(0,v,t)|^2\,dtdv.
\ea
If we notice
\bea
&&
\int_V\int_Q\left|\int_VK(x,v,v')u(x,v',t)\,dv'\right|^2e^{2s\alpha}\,dxdtdv
\nonumber \\
&&\le
C\int_V\int_Q\left(\int_V\left[|u(x,v',t)|^2+|\pp_xu(x,v',t)|^2\right]\,dv'
\right)e^{2s\alpha}\,dxdtdv
\nonumber \\
&&\le
C|V|\int_V\int_Q\bigl(|u(x,v,t)|^2+|\pp_xu(x,v,t)|^2\bigr)e^{2s\alpha}\,dxdtdv,
\label{ifwenotice}
\eea
we have
\ba
&&
\int_V\int_Q\left[\frac{1}{s\va}|\pp_tu|^2+
s\lambda^2\va|\pp_xu|^2+s^3\lambda^4\va^3|u|^2\right]e^{2s\alpha}\,dxdtdv
\\
&&\le
C\int_V\int_Q |f|^2 e^{2s\alpha}\,dxdtdv
+C\int_V\int_Q\left(|\pp_xu|^2+|u|^2\right)e^{2s\alpha}\,dxdtdv
\\
&&+
Ce^{C(\lambda)s}\int_0^T\int_{\Gamma_+}\left(|u|^2+|\pp_tu|^2+|\pp_xu|^2
\right)\,dSdvdt 
\\
&&
+Ce^{C(\lambda)s}\int_V \int_0^T|\pp_x u(0,v,t)|^2\,dtdv.
\ea
Taking sufficiently large $s>0$, we can absorb the second term on the 
right-hand side of the above inequality and we obtain the Carleman estimate 
(\ref{prop:bd:eq1}). Below we will derive (\ref{principal}).

Let us define
\[
Pz:=e^{s\alpha}L_0(e^{-s\alpha}z)=e^{s\alpha}L_0u.
\]
We split $Pz$ into three parts:
\[
Pz=P_1z+P_2z-R_0z,
\]
where
\ba
P_1z
&=&
-v^2\pp_x^2z-s^2\lambda^2\va^2(\pp_xd)^2v^2z-s(\pp_t\alpha)z,
\\
P_2z
&=&
\pp_tz+2s\lambda\va(\pp_xd)v^2\pp_xz,
\\
R_0z
&=&
-s\lambda^2\va(\pp_xd)^2v^2z-s\lambda\va(\pp_x^2d)v^2z.
\ea
We note that
\[
\|P_1z+P_2z\|_{L^2(Q\times V)}^2
\le2\|Pz\|_{L^2(Q\times V)}^2+2\|R_0z\|_{L^2(Q\times V)}^2.
\]
Here,
\[
\|P_1z+P_2z\|_{L^2(Q\times V)}^2
=\|P_1z\|_{L^2(Q\times V)}^2+\|P_2z\|_{L^2(Q\times V)}^2
+2\int_{Q\times V}(P_1z)(P_2z)\,dxdvdt.
\]
Therefore we have
\be
\frac{1}{2}\|P_2z\|_{L^2(Q\times V)}^2+\int_{Q\times V}(P_1z)(P_2z)\,dxdvdt
\le\|Pz\|_{L^2(Q\times V)}^2+\|R_0z\|_{L^2(Q\times V)}^2.
\label{eq:prfce03}
\ee
Let us calculate the left-hand side of the above inequality term by term. 
First, using the inequality $|z_1+z_2|^2\ge\frac{1}{2}|z_1|^2-|z_2|^2$ 
($z_1,z_2\in\Cm$), we have for any $\varepsilon\in(0,1]$,
\bea
\|P_2z\|^2_{L^2(Q\times V)}
&=&
\int_V\int_Q|P_2 z|^2\,dxdtdv
\nonumber \\
&\ge&
\varepsilon\int_V\int_Q\frac{1}{s\va}|P_2 z|^2\,dxdtdv
\nonumber \\
&\ge&
\frac{\varepsilon}{2}\int_V\int_Q\frac{1}{s\va}|\pp_tz|^2\,dxdtdv
-4\varepsilon v_1^4\int_V\int_Q s\lambda^2\va(\pp_xd)^2|\pp_xz|^2\,dxdtdv.
\nonumber \\
\label{eq:prfce04}
\eea
The second term can be estimated as follows. Let us write
\[
\int_{Q\times V}(P_1z)(P_2z)\,dxdvdt=I_1+I_2+I_3+I_4+I_5,
\]
where
\ba
I_1
&=&
\int_{Q\times V}\bigl(-v^2\pp_x^2z\bigr)\bigl(\pp_tz\bigr)\,dxdvdt,
\\
I_2
&=&
\int_{Q\times V}\bigl(-v^2\pp_x^2z\bigr)
\bigl(2s\lambda\va(\pp_xd)v^2\pp_xz\bigr)\,dxdvdt,
\\
I_3
&=&
\int_{Q\times V}\bigl(-s^2\lambda^2\va^2(\pp_xd)^2v^2z\bigr)
\bigl(\pp_tz\bigr)\,dxdvdt,
\\
I_4
&=&
\int_{Q\times V}\bigl(-s^2\lambda^2\va^2(\pp_xd)^2v^2z\bigr)
\bigl(2s\lambda\va(\pp_xd)v^2\pp_xz\bigr)\,dxdvdt,
\\
I_5
&=&
\int_{Q\times V}\bigl(-s(\pp_t\alpha)z\bigr)
\bigl(\pp_tz+2s\lambda\va(\pp_xd)v^2\pp_xz\bigr)\,dxdvdt.
\ea
We can compute $I_1$ through $I_5$ using integration by parts and the 
Schwarz inequality. Note that $z(x,v,t)=\pp_tz(x,v,t)=0$ in 
$\Gamma_- \times (0,T)$ because $u(x,v,t)=0$, 
$(x,v,t)\in \Gamma_- \times (0,T)$. We have
\be
I_1
=-\int_0^T\left(\left.\int_{v_0}^{v_1}v^2(\pp_xz)(\pp_tz)\,dv\right|_{x=\ell}
+\left.\int_{-v_1}^{-v_0}v^2(\pp_xz)(\pp_tz)\,dv\right|_{x=0}\right)\,dt.
\label{eq:I1}
\ee
For the second term, there exists $C>0$ such that
\bea
I_2
&=&
-\int_V\int_Qs\lambda\va(\pp_xd)v^4\pp_x|\pp_xz|^2\,dxdtdv
\nonumber \\
&\ge&
\int_V\int_Qs\lambda^2\va(\pp_xd)^2v^4|\pp_xz|^2\,dxdtdv-
C\int_V\int_Qs\lambda\va|\pp_xz|^2\,dxdtdv
\nonumber \\
&-&
\int_0^T\int_V\left(\left.s\lambda\va(\pp_xd)v^4|\pp_xz|^2\right|_{x=\ell}
+\left.s\lambda\va(\pp_xd)v^4|\pp_xz|^2\right|_{x=0}\right)\,dvdt.
\label{eq:I2}
\eea
We can estimate the third term as
\be
I_3
=
-\frac{1}{2}\int_V\int_Qs^2\lambda^2\va^2(\pp_xd)^2v^2\pp_t|z|^2\,dxdtdv
\ge
-C\int_V\int_Qs^2\lambda^2\va^3|z|^2\,dxdtdv.
\label{eq:I3}
\ee
The fourth term is estimated as
\bea
I_4
&=&
-\int_V\int_Qs^3\lambda^3\va^3(\pp_xd)^3v^4\pp_x|z|^2\,dxdtdv
\nonumber \\
&\ge&
3\int_V\int_Qs^3\lambda^4\va^3(\pp_xd)^4v^4|z|^2\,dxdtdv
-C\int_V\int_Qs^3\lambda^3\va^3|z|^2\,dxdtdv
\nonumber \\
&-&
\int_0^Ts^3\lambda^3\left(
\left.\int_{v_0}^{v_1}\va^3(\pp_xd)^3v^4|z|^2\,dv\right|_{x=\ell}+
\left.\int_{-v_1}^{-v_0}\va^3(\pp_xd)^3v^4|z|^2\,dv\right|_{x=0}
\right)\,dt.
\nonumber \\
\label{eq:I4}
\eea
The last term $I_5$ is computed as
\bea
&&
I_5
=
-\frac{1}{2}\int_V\int_Qs(\pp_t\alpha)\pp_t|z|^2\,dxdtdv
-\int_V\int_Qs^2\lambda\va(\pp_xd)(\pp_t\alpha)v^2\pp_x|z|^2\,dxdtdv
\nonumber \\
&&\ge
-C(\lambda)\int_V\int_Q(s\va^3+s^2\va^3)|z|^2\,dxdtdv
\nonumber \\
&&-
\int_0^Ts^2\lambda\left(
\left.\int_{v_0}^{v_1}\va(\pp_xd)(\pp_t\alpha)v^2|z|^2\,dv\right|_{x=\ell}+
\left.\int_{-v_1}^{-v_0}\va(\pp_xd)(\pp_t\alpha)v^2|z|^2\,dv\right|_{x=0}
\right)\,dt.
\nonumber \\
\label{eq:I5}
\eea
By putting (\ref{eq:I1}) through (\ref{eq:I5}) together, we obtain
\bea
&&
\int_V\int_Qs\lambda^2\va(\pp_xd)^2v^4|\pp_xz|^2\,dxdtdv
+3\int_V\int_Qs^3\lambda^4\va^3(\pp_xd)^4v^4|z|^2\,dxdtdv
\nonumber \\
&&\le
\int_{Q\times V}(P_1z)(P_2 z)\,dxdvdt
+C\int_V\int_Qs\lambda\va|\pp_xz|^2\,dxdtdv
\nonumber \\
&&+
\int_V\int_Q(s^3\lambda^3\va^3+s^2\lambda^2\va^3)|z|^2\,dxdtdv
+C(\lambda)\int_V\int_Q(s\va^3+s^2\va^3)|z|^2\,dxdtdv
\nonumber \\
&&+
B,
\label{eq:III}
\eea
where
\ba
B
&=&
\int_0^T\left(
\left.\int_{v_0}^{v_1}v^2(\pp_xz)(\pp_tz)\,dv\right|_{x=\ell}+
\left.\int_{-v_1}^{-v_0}v^2(\pp_xz)(\pp_tz)\,dv\right|_{x=0}
\right)\,dt
\\
&+&
\int_0^T\int_V\left(
\left.s\lambda\va(\pp_xd)v^4|\pp_xz|^2\right|_{x=\ell}+
\left.s\lambda\va(\pp_xd)v^4|\pp_xz|^2\right|_{x=0}
\right)\,dvdt
\\
&+&
\int_0^T\left(
\left.\int_{v_0}^{v_1}s^3\lambda^3\va^3(\pp_xd)^3v^4|z|^2\,dv\right|_{x=\ell}+
\left.\int_{-v_1}^{-v_0}s^3\lambda^3\va^3(\pp_xd)^3v^4|z|^2\,dv\right|_{x=0}
\right)\,dt
\\
&+&
\int_0^Ts^2\lambda\left(
\left.\int_{v_0}^{v_1}\va(\pp_xd)(\pp_t\alpha)v^2|z|^2\,dv\right|_{x=\ell}+
\left.\int_{-v_1}^{-v_0}\va(\pp_xd)(\pp_t\alpha)v^2|z|^2\,dv\right|_{x=0}
\right)\,dt.
\ea
The remainder term $\|R_0z\|^2_{L^2 (Q\times V)}$ is estimated as follows.
\be
\|R_0z\|^2_{L^2 (Q\times V)}
\le
C\int_V\int_Q(s^2\lambda^4\va^2+s^2\lambda^2\va^2)|z|^2\,dxdtdv.
\label{eq:RR}
\ee
Let us apply the estimates (\ref{eq:prfce04}), (\ref{eq:III}), (\ref{eq:RR}) 
in (\ref{eq:prfce03}). For sufficiently small $\varepsilon$, we have
\bea
&&
\int_V\int_Q\left[\frac{1}{s\va}|\pp_tz|^2+s\lambda^2\va|\pp_xz|^2+
s^3\lambda^4\va^3|z|^2\right]\,dxdtdv
\nonumber \\
&&\le
C\int_V\int_Q|Pz|^2\,dxdtdv+
C\int_V\int_Qs\lambda\va|\pp_xz|^2\,dxdtdv
\nonumber \\
&&+
C\int_V\int_Q(s^2\lambda^4\va^2+s^3\lambda^3\va^3+
s^2\lambda^2\va^3+s^2\lambda^2\va^2)|z|^2\,dxdtdv
\nonumber \\
&&+
C(\lambda)\int_V\int_Q(s^2\va^3+s\va^3)|z|^2\,dxdtdv+CB.
\label{eq:prfce08}
\eea
The boundary term is estimated as
\bea
B
&\le&
Ce^{C(\lambda)s}\int_0^T\int_{\Gamma_+}\left(|z|^2+|\pp_tz|^2+
|\pp_xz|^2\right)\,dSdvdt
\nonumber\\
&+&
Ce^{C(\lambda)s}\int_0^T\int_{V}|\pp_xz(0,v,t)|^2dvdt.
\label{eq:bdry}
\eea
Therefore for sufficiently large $s,\lambda$, we obtain
\bea
&&
\int_V\int_Q\left[\frac{1}{s\va}|\pp_tu|^2
+s\lambda^2\va|\pp_xu|^2+s^3\lambda^4\va^3|u|^2\right]e^{2s\alpha}
\,dxdtdv
\nonumber \\
&&\le
C\int_V\int_Q|L_0u|^2e^{2s\alpha}\,dxdtdv
+C\int_V\int_Q s\lambda\va|\pp_xu|^2e^{2s\alpha}\,dxdtdv
\nonumber \\
&&+
C\int_V\int_Q(s^2\lambda^4\va^2+s^3\lambda^3\va^3)|u|^2e^{2s\alpha}\,dxdtdv
+C(\lambda)\int_V\int_Q s^2\va^3|u|^2e^{2s\alpha}\,dxdtdv
\nonumber \\
&&+
Ce^{C(\lambda)s}\int_0^T\int_{\Gamma_+}\bigl(|u|^2+|\pp_tu|^2+|\pp_xu|^2
\bigr)\,dSdvdt
\nonumber\\
&&
+Ce^{C(\lambda)s}\int_0^T\int_{V}|\pp_x u(0,v,t)|^2dvdt.
\label{eq:prfce12}
\eea
The second, third and fourth terms on the right-hand side of 
(\ref{eq:prfce12}) can be absorbed in the left-hand side, and 
(\ref{principal}) is derived. Thus the proof is complete.
\end{proof}

\begin{rmk}
The proof is similar to the calculation in 
\cite{Emanuilov95, Fursikov-Imanuvilov96, Yamamoto09} 
in the sense that the same weight function is used. However, our equation 
contains the integral term, and furthermore the surface integral appears in 
the Carleman estimate due to the half-range boundary condition in (\ref{rte4}).
\end{rmk}

\section{
Proofs of Theorems \ref{thm:bd}, \ref{thm:bdt}, and \ref{thm:bds}
}

\subsection{Proof of Theorem \ref{thm:bd}}
\label{proofbd}

Here we prove Theorem \ref{thm:bd} by making use of Proposition \ref{prop:bd}.

Let us recall that $\vv{u}$ satisfies (\ref{rte3}). We set
\[
\vv{y}(x,v,t)=\pp_t\vv{u}(x,v,t).
\]
We obtain
\be
\pp_t\vv{y}=v^2\pp_x^2\vv{y}+L_1\vv{y}+\int_VK(x,v,v')\vv{y}(x,v',t)\,dv'+
\pp_t\vv{f}(x,v,t),
\label{zeq}
\ee
where each component of $\vv{y}$ satisfies $y_j(x,v,t)=0$ on 
$\Gamma_-\times(0,T)$ ($j=1,2$). For $0<t_0<T$, we have
\be
f_j(x,v,t_0)=y_j(x,v,t_0)-v^2\pp_x^2u_j(x,v,t_0)-L_1u_j(x,v,t_0)
-\int_VK(x,v,v')u_j(x,v',t_0)\,dv',
\label{zeqcond1}
\ee
for $j=1,2$.

We consider the Carleman estimate for (\ref{zeq}) on $Q_\delta$. 
We here use the following weight function for the Carleman estimate instead of 
$\alpha(x,t)$ in (\ref{weightalpha}). 
\[
\alpha_\delta(x,t)=
\frac{e^{\lambda d(x)}-e^{2\lambda\|d\|_{C(\overline{\Omega})}}}{(t-t_0+\delta)(t_0+\delta-t)}.
\]
We define
\[
\va_\delta(x,t)=\frac{e^{\lambda d(x)}}{(t-t_0+\delta)(t_0+\delta-t)}.
\]
We can readily see that Proposition \ref{prop:bd} holds true for 
$t\in(t_0-\delta,t_0+\delta)$ instead of $t\in(0,T)$. 
For a sufficiently large fixed $\lambda>0$, we can write the Carleman 
estimate in Proposition \ref{prop:bd} as
\bea
&&
\int_V\int_{Q_\delta}\left[\frac{1}{s\va_{\delta}} |\pp_ty_j|^2+
s\va_{\delta}|\pp_xy_j|^2+s^3\va_{\delta}^3|y_j|^2\right]
e^{2s\alpha_{\delta}}\,dxdtdv
\nonumber \\
&&\le 
C\int_V\int_{Q_\delta}|\pp_tf_j|^2e^{2s\alpha_{\delta}}\,dxdtdv+
Ce^{Cs}\int_{t_0-\delta}^{t_0+\delta}\int_{\Gamma_+}
\left(|y_j|^2+|\pp_ty_j|^2+|\pp_xy_j|^2\right)\,dSdvdt 
\nonumber \\
&&+
Ce^{C s}\int_V\int_{t_0-\delta}^{t_0+\delta}|\pp_xy_j(0,v,t)|^2\,dtdv,
\label{zce}
\eea
for $j=1,2$.

To estimate $\int_{\Omega\times V}|\pp_tu_j(x,v,t_0)|^2
e^{2s\alpha_{\delta}(x,t_0)}\,dxdv$ from above, we note 
that
\[
\lim_{t\to t_0-\delta+0}e^{2s\alpha_{\delta}(x,t)}=0
\qquad\mbox{for}\quad x\in \Omega.
\]
Hence we have
\ba
\int_{\Omega\times V}\left|y_j(x,v,t_0)\right|^2e^{2s\alpha_\delta (x,t_0)}\,
dxdv
&=&
\int_{t_0-\delta}^{t_0}\pp_t\left(\int_{\Omega\times V}|y_j(x,v,t)|^2
e^{2s\alpha_{\delta}(x,t)}\,dxdv\right)\,dt
\\
&=&
\int_{\Omega\times V}\int_{t_0-\delta}^{t_0}\left(2|y_j||\pp_ty_j|+
2s(\pp_t\alpha_\delta)|y_j|^2\right)e^{2s\alpha_{\delta}(x,t)}\,dtdxdv.
\ea
We can further estimate the above inequality by noting that 
$|\pp_t\alpha_\delta|\le C\va_\delta^2$ and using
\[
|y_j||\pp_ty_j|
=\Biggl(\Biggr. \frac{1}{s\sqrt{\va_\delta}}|\pp_ty_j|\Biggl.\Biggr) 
\Bigl(\Bigr. s\sqrt{\va_\delta}|y_j| \Bigl.\Bigr)
\le\frac{1}{2s^2\va_\delta}|\pp_ty_j|^2+\frac{1}{2}s^2\va_\delta|y_j|^2,
\]
and applying (\ref{zce}). We obtain
\ba
&&
\int_{\Omega\times V}\left|y_j(x,v,t_0)\right|^2e^{2s\alpha_\delta(x,t_0)}
\,dxdv
\\
&&\le
C\int_{V}\int_{Q_\delta}
\left(\frac{1}{s^2\va_\delta}|\pp_ty_j|^2+s^2\va_\delta^2|y_j|^2\right)
e^{2s\alpha_\delta(x,t)}\,dxdtdv
\\
&&\le
\frac{C}{s}\int_{V}\int_{Q_\delta}|\pp_tf|^2e^{2s\alpha_\delta(x,t)}\,dxdtdv+
Ce^{Cs}\int_{t_0-\delta}^{t_0+\delta}\int_{\Gamma_+}
\left(|y_j|^2+|\pp_ty_j|^2+|\pp_xy_j|^2\right)\,dSdvdt
\\
&&+
Ce^{Cs}\int_V\int_{t_0-\delta}^{t_0+\delta}|\pp_xy_j(0,v,t)|^2\,dtdv.
\ea
That is,
\ba
\int_{\Omega\times V}\left|\pp_tu_j(x,v,t_0)\right|^2
e^{2s\alpha_\delta(x,t_0)}\,dxdv
&\le&
\frac{C}{s}\int_{V}\int_{Q_\delta}|\pp_tf_j|^2e^{2s\alpha_\delta(x,t)}\,dxdtdv
\\
&+&
Ce^{Cs}\int_{t_0-\delta}^{t_0+\delta}\int_{\Gamma_+}
\left(|\pp_tu_j|^2+|\pp_t^2u_j|^2+|\pp_x\pp_tu_j|^2\right)\,dSdvdt
\\
&+&
Ce^{Cs}\int_V\int_{t_0-\delta}^{t_0+\delta}|\pp_x\pp_tu_j(0,v,t)|^2\,dtdv,
\ea
for $j=1,2$. By taking the weighted $L^2$ norm of (\ref{zeqcond1}) using the 
above inequality, we obtain
\bea
&&
\int_{\Omega\times V}\left|f_j(x,v,t_0)\right|^2e^{2s\alpha_\delta(x,t_0)}
\,dxdv
\nonumber \\
&&\le
\int_{\Omega\times V}\left|\pp_tu_j(x,v,t_0)\right|^2
e^{2s\alpha_\delta(x,t_0)}\,dxdv
+Ce^{Cs}\left\|u_j(\cdot,\cdot,t_0)\right\|_{H^2(\Omega;L^2(V))}^2
\nonumber \\
&&\le
\frac{C}{s}\int_{\Omega\times V}|\pp_tf_j|^2e^{2s\alpha_\delta(x,t)}\,dxdvdt
+Ce^{Cs}\left\|u_j(\cdot,\cdot,t_0)\right\|_{H^2(\Omega;L^2(V))}^2
\nonumber \\
&&+
Ce^{Cs}\int_{t_0-\delta}^{t_0+\delta}\int_{\Gamma_+}
\left(|\pp_tu_j|^2+|\pp_t^2u_j|^2+|\pp_x\pp_tu_j|^2\right)\,dSdvdt
\nonumber \\
&&+
Ce^{Cs}\int_V\int_{t_0-\delta}^{t_0+\delta}|\pp_x\pp_tu_j(0,v,t)|^2\,dtdv,
\label{fe01}
\eea
where the integral term on the right-hand side of (\ref{zeqcond1}) was 
estimated by a calculation similar to (\ref{ifwenotice}). By differentiating 
(\ref{defoff}), we obtain
\ba
&&
\pp_t\vv{f}(x,v,t)=
-v\pp_tR(x,v,t)\pp_x\vv{r}(x,v)
\\
&&
-\left[v\pp_t\pp_xR(x,v,t)+\sigma_t(x,v)\pp_tR(x,v,t)-
\pp_t\pp_t^{1/2}R(x,v,t)+\frac{1}{2\Gamma(\frac{1}{2})t\sqrt{t}}R(x,v,0)\right]
\vv{r}(x,v)
\\
&&
+\sigma_s(x,v)\int_Vp(x,v,v^\prime)\pp_tR(x,v^\prime,t)\vv{r}(x,v^\prime)
\,dv^\prime.
\ea
Thus we have
\bea
&&
\int_V\int_{Q_\delta}|\pp_tf_j|^2e^{2s\alpha_{\delta}(x,t)}\,dxdtdv 
\nonumber \\
&&\leq
C\int_V\int_{Q_\delta}\left(|r_t|^2+|\pp_xr_t|^2+|r_s|^2+|\pp_xr_s|^2\right)
e^{2s\alpha_{\delta}(x,t)}\,dxdtdv
\nonumber \\
&&\leq
C\int_{\Omega\times V}\left(|r_t|^2+|\pp_xr_t|^2+|r_s|^2+|\pp_xr_s|^2\right)
e^{2s\alpha_{\delta}(x,t_0)}\,dxdv.
\label{fte}
\eea
Here, since $\alpha_\delta(x,t)\le\alpha_\delta(x,t_0)$ was used, $C$ depends on $t_0$ and $\delta$. By (\ref{fe01}) and (\ref{fte}), we obtain
\bea
&&
\int_{\Omega\times V}\left|f_j(x,v,t_0)\right|^2e^{2s\alpha_\delta(x,t_0)}\,
dxdv
\nonumber \\
&&\le
\frac{C}{s}
\int_{\Omega\times V}
\left(|\pp_xr_t|^2+|r_t|^2+|\pp_xr_s|^2+|r_s|^2\right)
e^{2s\alpha_{\delta}(x,t_0)}\,dxdv
+Ce^{Cs}\left\|u_j(\cdot,\cdot,t_0)\right\|_{H^2(\Omega;L^2(V))}^2
\nonumber \\
&&+
Ce^{Cs}\int_{t_0-\delta}^{t_0+\delta}\int_{\Gamma_+}
\left(|\pp_tu_j|^2+|\pp_t^2u_j|^2+|\pp_x\pp_tu_j|^2\right)\,dSdvdt
\nonumber \\
&&+
Ce^{Cs}\int_V\int_{t_0-\delta}^{t_0+\delta}|\pp_x\pp_tu_j(0,v,t)|^2\,dtdv,
\label{fe02}
\eea
where $j=1,2$.

Let us estimate $\int_{\Omega\times V}|f_j(x,v,t_0)|^2
e^{2s\alpha_{\delta}(x,t_0)}\,dxdv$ from below in terms of $r_t$ and $r_s$. 
For this purpose we use the following proposition.

\begin{prop}
Suppose $\vv{w}(x,v)$ satisfies
\[
\pp_x\vv{w}(x,v)+A(x,v)\vv{w}(x,v)+\int_VD(x,v,v')\vv{w}(x,v')\,dv'
=\vv{F}(x,v),
\]
where $A\in L^\infty(\Omega\times V)^{2\times2}$ and 
$D\in L^\infty(\Omega\times V\times V)^{2\times2}$. Then for sufficiently 
large $s>0$, there exists a constant $C>0$ such that
\[
\int_{\Omega\times V} 
\left[\left|\pp_x\vv{w}(x,v)\right|^2+s^2\left|\vv{w}(x,v)\right|^2\right]
e^{2s\alpha_\delta (x,t_0)}\,dxdv
\le
C\int_{\Omega\times V}\left|\vv{F}(x,v)\right|^2e^{2s\alpha_\delta(x,t_0)}
\,dxdv,
\]
for all $\vv{w}\in H^1(\Omega;L^2(V))^2$ and $\vv{w}(0,v)=\vv{0}$, $v\in V$. 
\end{prop}

\begin{proof}
Let us express $A$, $D$, $\vv{w}$, and $\vv{F}$ as
\[
A=\left(\begin{array}{cc}
A_{11} & A_{12} \\ A_{21} & A_{22}
\end{array}\right),
\qquad
D=\left(\begin{array}{cc}
D_{11} & D_{12} \\ D_{21} & D_{22}
\end{array}\right),
\qquad
\vv{w}=\left(\begin{array}{c}w_1 \\ w_2\end{array}\right),
\qquad
\vv{F}=\left(\begin{array}{c}F_1 \\ F_2\end{array}\right).
\]
We have
\ba
&&
\pp_xw_1(x,v)+A_{11}(x,v)w_1(x,v)+\int_VD_{11}(x,v,v')w_1(x,v')\,dv'
=\tilde{F}_1(x,v),
\\
&&
\pp_xw_2(x,v)+A_{22}(x,v)w_2(x,v)+\int_VD_{22}(x,v,v')w_2(x,v')\,dv'
=\tilde{F}_2(x,v),
\ea
where
\ba
&&
\tilde{F}_1(x,v)
=F_1(x,v)-A_{12}(x,v)w_2(x,v)-\int_VD_{12}(x,v,v')w_2(x,v')\,dv',
\\
&&
\tilde{F}_2(x,v)
=F_2(x,v)-A_{21}(x,v)w_1(x,v)-\int_VD_{21}(x,v,v')w_1(x,v')\,dv'.
\ea
If we use Lemma \ref{lemmace1} in Appendix, we obtain
\ba
&&
\int_{\Omega\times V}\left(|\pp_xw_1(x,v)|^2+s^2|w_1(x,v)|^2+
|\pp_xw_2(x,v)|^2+s^2|w_2(x,v)|^2\right)e^{2s\alpha_{\delta}(x,t_0)}\,dxdv
\\
&&\le
C\int_{\Omega\times V}\left(\left|\tilde{F}_1(x,v)\right|^2+
\left|\tilde{F}_2(x,v)\right|^2\right)e^{2s\alpha_{\delta}(x,t_0)}\,dxdv
\\
&&\le
C\int_{\Omega\times V}\left(\left|F_1(x,v)\right|^2+
\left|F_2(x,v)\right|^2\right)\,dxdv
+C\int_{\Omega\times V}\left(|w_1(x,v)|^2+|w_2(x,v)|^2\right)\,dxdv.
\ea
The proof is complete by noticing that terms 
$\int_V|w_j(x,v)|^2\,dv$ ($j=1,2$) can be absorbed to the left-hand side 
if $s$ is sufficiently large.
\end{proof}

Recall that we assumed $\det{R(x,v,t_0)}\neq0$ and $|v|>v_0>0$. We apply 
the above Proposition after rewriting (\ref{defoff}) as
\ba
&&
\pp_x\vv{r}(x,v)+
\frac{1}{vR(x,v,t_0)}
\\
&&\times
\left(v\pp_xR(x,v,t_0)+\sigma_t(x,v)R(x,v,t_0)-\pp_t^{1/2}R(x,v,t_0)-
\frac{1}{\Gamma\left(\frac{1}{2}\right)t^{1/2}}R(x,v,0)\right)\vv{r}(x,v)
\\
&&+
\int_V\left(\frac{-\sigma_s(x,v)}{vR(x,v,t_0)}p(x,v,v')R(x,v',t_0)\right)
\vv{r}(x,v')\,dv'
\\
&&=
\frac{-1}{vR(x,v,t_0)}\vv{f}(x,v,t_0).
\ea
By $1/R$ we denote the inverse matrix of $R$, that is, $1/R=R^{-1}$. We obtain
\bea
&&
\int_{\Omega\times V}\left(|\pp_xr_t|^2+s^2|r_t|^2+|\pp_xr_s|^2+s^2|r_s|^2
\right)e^{2s\alpha_{\delta}(x,t_0)}\,dxdv
\nonumber \\
&&\le
C\sum_{j=1}^2\int_{\Omega\times V}|f_j(x,v,t_0)|^2e^{2s\alpha_{\delta}(x,t_0)}
\,dxdv.
\label{inu}
\eea
Thus we have
\bea
&&
\left(1-\frac{C}{s}\right)
\int_{\Omega\times V}
\left(
|\pp_x r_t|^2+|r_t|^2
+
|\pp_x r_s|^2+|r_s|^2
\right)e^{2s\alpha_{\delta}(x,t_0)}
\,dxdv
\nonumber \\
&&\le
Ce^{Cs}
\sum_{j=1}^2
\left\|u_j(\cdot,\cdot,t_0)\right\|_{H^2(\Omega;L^2(V))}^2
\nonumber \\
&&+
Ce^{Cs}
\sum_{j=1}^2
\int_{t_0-\delta}^{t_0+\delta}\int_{\Gamma_+}
\left(|\pp_t u_j|^2+|\pp_t^2 u_j|^2+|\pp_x \pp_t u_j|^2 \right)
\,dSdvdt
\nonumber \\
&&+
Ce^{C s}
\sum_{j=1}^2
\int_V \int_{t_0-\delta}^{t_0+\delta} |\pp_x \pp_t u_j(0,v,t)|^2 \,dtdv.
\nonumber
\eea
If we take sufficiently large $s>0$,
we obtain the stability estimate in Theorem \ref{thm:bd}.

\hfill\qed

\subsection{Proof of Theorem \ref{thm:bdt}}

Instead of the vector-valued $\vv{r}(x,v),\vv{u}(x,v,t)$, we introduce
\[
r(x,v)=r_t(x,v),\qquad u(x,v,t)=u^{(1)}(x,v,t)-u^{(2)}(x,v,t).
\]
Correspondingly, we have $R(x,v,t)=-u_1^{(2)}(x,v,t)$. We can carry out almost the identical calculation in \S\ref{half2one} and \S\ref{proofbd} with these $r(x,v),u(x,v,t)$. As a result, we similarly obtain the stability estimate in Theorem \ref{thm:bdt}.

\hfill\qed

\subsection{Proof of Theorem \ref{thm:bds}}

Instead of the vector-valued $\vv{r}(x,v),\vv{u}(x,v,t)$, we introduce
\[
r(x,v)=r_s(x,v),\qquad u(x,v,t)=u^{(1)}(x,v,t)-u^{(2)}(x,v,t).
\]
In this case, we have $R(x,v,t)=\int_Vp(x,v,v')u_1^{(2)}(x,v',t)\,dv'$. We can carry out almost the same calculation in \S\ref{half2one} and \S\ref{proofbd} with these $r(x,v),u(x,v,t)$. As a result, we similarly obtain the stability estimate in Theorem \ref{thm:bds}.

\hfill\qed

\section*{Acknowledgement(s)}

This work was supported by the interdisciplinary project on environmental transfer of radionuclides from University of Tsukuba and Hirosaki University, and by the A3 foresight program ``Modeling and Computation of Applied Inverse Problems'' of the Japan Society of the Promotion of Science. MM also acknowledges support from Grant-in-Aid for Scientific Research (17K05572 and 17H02081) of the Japan Society for the Promotion of Science (JSPS).

\appendix

\section{}

Let us consider
\be
\pp_xw(x,v)+b(x,v)w(x,v)+\int_Vc(x,v,v^\prime)w(x,v^\prime)\,dv^\prime=F(x,v),
\label{eqw}
\ee
where $b\in L^\infty(\Omega\times V)$ and 
$c\in L^\infty(\Omega\times V\times V)$.

\begin{lem}
\label{lemmace1}
For sufficiently large $s>0$, there exists a constant $C>0$ such that
\[
\int_{\Omega\times V} 
\left[\left|\pp_x w (x,v)\right|^2 + s^2 \left|w (x,v)\right|^2\right]
e^{2s\alpha_\delta (x,t_0)}\,dxdv
\le
C\int_{\Omega\times V}\left|F(x,v)\right|^2e^{2s\alpha_\delta(x,t_0)}\,dxdv,
\]
for all $w \in H^1(\Omega;L^2(V))$ satisfying (\ref{eqw}) and $w(0,v)=0$, 
$v\in V$. 
\end{lem}

\begin{proof}
Hereafter we let $C$ denote generic constants which do not depend on 
$s$ but may depend on $\lambda$. 

Let us set $\widetilde{w}=we^{s\alpha_\delta(\cdot,t_0)}$ and define 
$\widetilde{P}$ by
\[
\widetilde{P}\widetilde{w}
=e^{s\alpha_\delta(\cdot,t_0)}\pp_x\left(\widetilde{w}
e^{-s\alpha_\delta(\cdot,t_0)}\right).
\]
Then we have
\[
\widetilde{P}\widetilde{w}
=\pp_x\widetilde{w}-s\lambda\va_{\delta}(\cdot,t_0)(\pp_xd)\widetilde{w}.
\]
Taking $L^2$-norm for $\widetilde{P}\widetilde{w}$, we obtain
\ba
&&
\left\|\widetilde{P}\widetilde{w}\right\|_{L^2(\Omega\times V)}^2
=
\|\pp_x\widetilde{w}\|_{L^2(\Omega\times V)}^2
+\|s\lambda\va_{\delta}(\cdot,t_0)(\pp_x d)\widetilde{w}
\|_{L^2(\Omega\times V)}^2
\\
&&
-2\int_{Q\times V}\left(\pp_x\widetilde{w}\right)
\left(s\lambda\va_{\delta}(\cdot,t_0)(\pp_xd)\widetilde{w}\right)\,dxdv
\\
&&
\ge
C\int_{\Omega\times V}\left(|\pp_x\widetilde{w}|^2+s^2|\widetilde{w}|^2\right)
\,dxdv
-\int_{\Omega\times V}s\lambda\va_{\delta}(\cdot,t_0)(\pp_xd)\pp_x
|\widetilde{w}|^2\,dxdv
\\
&&
\ge
C\int_{\Omega\times V}\left(|\pp_x\widetilde{w}|^2+s^2|\widetilde{w}|^2\right)
\,dxdv
-C\int_{\Omega\times V}s|\widetilde{w}|^2\,dxdv,
\ea
where we could drop the boundary term which arose from integration by parts 
because $\pp_x d <0$ in $\overline{\Omega}$ and $w(0,\cdot)=0$ in $V$. 
Hence we have
\be
\int_{\Omega\times V}\left(|\pp_x\widetilde{w}|^2+s^2|\widetilde{w}|^2\right)
\,dxdv
\le
C\left\|\widetilde{P}\widetilde{w}\right\|_{L^2(\Omega\times V)}^2
+\int_{\Omega\times V}s|\widetilde{w}|^2\,dxdv.
\label{lem04}
\ee
Taking sufficiently large $s>0$, we may absorb the second term on the 
right-hand side of (\ref{lem04}) and we have
\[
\int_{\Omega\times V}\left(|\pp_x\widetilde{w}|^2+s^2|\widetilde{w}|^2\right)
\,dxdv
\le
C\left\|\widetilde{P}\widetilde{w}\right\|_{L^2(\Omega\times V)}^2.
\]
From the above equation for $\tilde{w}$, we arrive at the following inequality for $w$.
\bea
&&
\int_{\Omega\times V} 
\left[\left|\pp_x w (x,v)\right|^2 + s^2 \left|w (x,v)\right|^2\right]
e^{2s\alpha_\delta (x,t_0)}\,dxdv
\le
C\int_{\Omega\times V}\left|\pp_x w\right|^2e^{2s\alpha_\delta(x,t_0)}\,dxdv.
\nonumber \\
\label{lem01}
\eea

Since
\[
|\pp_x w|^2
\le
C|F|^2+C|w|^2+C\left|\int_{V}c(x,v,v^\prime)w(x,v^\prime)\,dv^\prime\right|^2,
\]
we obtain
\bea
&&
\int_{\Omega\times V} 
\left[\left|\pp_x w (x,v)\right|^2 + s^2 \left|w (x,v)\right|^2\right]
e^{2s\alpha_\delta (x,t_0)}\,dxdv
\nonumber \\
&&\le
C\int_{\Omega\times V}\left|F\right|^2e^{2s\alpha_\delta(x,t_0)}\,dxdv+
C\int_{\Omega\times V}\left|w \right|^2e^{2s\alpha_\delta(x,t_0)}\,dxdv
\nonumber \\
&&+
C\int_{\Omega\times V}\left|\int_{V}c(x,v,v^\prime)w(x,v^\prime)\,dv^\prime
\right|^2e^{2s\alpha_\delta(x,t_0)}\,dxdv
\nonumber \\
&&\le
C\int_{\Omega\times V}\left|F\right|^2e^{2s\alpha_\delta(x,t_0)}\,dxdv+
C\int_{\Omega\times V}\left|w \right|^2e^{2s\alpha_\delta(x,t_0)}\,dxdv,
\label{lem03}
\eea
where we noted that $c\in L^\infty(\Omega\times V\times V)$ and used the fact that, by the Schwarz inequality,
\[
\int_{\Omega\times V}\left|\int_{V} c(x,v,v^\prime)w(x,v^\prime)\,dv^\prime
\right|^2e^{2s\alpha_\delta(x,t_0)}\,dxdv
\le
C\int_{\Omega\times V}\left|w\right|^2e^{2s\alpha_\delta(x,t_0)}\,dxdv.
\]
Taking sufficiently large $s>0$, we can absorb the second term on the right-hand side of (\ref{lem03}) to the left-hand side. Thus we obtain the estimate in Lemma \ref{lemmace1}.
\end{proof}


\begin{thebibliography}{99}

\bibitem{Acosta15}
Acosta S.
Recovery of the absorption coefficient in radiative transport from a single measurement
Inv Prob Imag. 2015;9:289--300.

\bibitem{Adams-Gelhar92}
Adams E E, Gelhar L W.
Field study of dispersion in a heterogeneous aquifer 2. Spatial moments analysis.
Water Res Res. 1992;28:3293--3307.

\bibitem{Arridge-Schotland09}
Arridge S R, Schotland J C.
Optical tomography: forward and inverse problems.
Inverse Problems. 2009;25:123010.

\bibitem{Bal09}
Bal G.
Inverse transport theory and applications.
Inverse Problems. 2009;25:053001.

\bibitem{Bukhgeim-Klibanov81}
Bukhgeim A L, Klibanov M V.
Global uniqueness of a class of multidimensional inverse problems.
Soviet Math Dokl. 1981;24:244--7.

\bibitem{Caputo67}
Caputo M.
Linear model of dissipation whose $Q$ is almost frequency independent-II.
Geophys J R Astr Soc. 1967;13:529--539.

\bibitem{Carleman39}
Carleman T.
Sur un probl\`{e}me d'unicit\'{e} pour les syst\`{e}mes d'\'{e}quations aux
d\'{e}riv\'{e}es partielles \`{a} deux variables ind\'{e}pendantes.
Ark Mat Astr Fys. 1939;26B:1--9.

\bibitem{Cheng-etal09}
Cheng J, Nakagawa J, Yamamoto M, Yamazaki T.
Uniqueness in an inverse problem for a one-dimensional fractional diffusion equation.
Inverse Problems. 2009;25:115002.

\bibitem{Emanuilov95}
\`{E}manuilov O Yu.
Controllability of parabolic equations.
Sbornik Math. 1995;186:879--900.

\bibitem{Fursikov-Imanuvilov96}
Fursikov A V, Imanuvilov O Y.
Controllability of Evolution Equations. (Lecture Notes Series vol.~34).
Seoul (Korea): Seoul National University; 1996.

\bibitem{Gaitan-Ouzzane13}
Gaitan P, Ouzzane H.
Inverse problem for a free transport equation using Carleman estimates.
Appl Anal. 2013;93:1073--1086.

\bibitem{Hatano-Hatano98}
Hatano Y, Hatano N.
Dispersive transport of ions in column experiments: An explanation of long-tailed profiles.
Water Resour Res. 1998;134:1027--1033.

\bibitem{Hatano-etal13}
Hatano Y, Nakagawa J, Wang S, Yamamoto M.
Determination of order in fractional diffusion equation.
J Math-for-Industry. 2013;5:51--57.

\bibitem{Imanuvilov-Yamamoto98}
Imanuvilov O Yu, Yamamoto M.
Lipschitz stability in inverse parabolic problems by the Carleman estimate.
Inverse Problems. 1998;14:1229--1245.

\bibitem{Jin-Rundell12}
Jin B, Rundell W.
An inverse problem for a one-dimensional time-fractional diffusion problem.
Inverse Problems. 2012;28:075010.

\bibitem{Klibanov13}
Klibanov M V.
Carleman estimates for global uniqueness, stability and numerical methods for coefficient inverse problems.
J Inverse Ill-Posed Probl. 2013;21:477--560.

\bibitem{Klibanov-Pamyatnykh06}
Klibanov M V, Pamyatnykh S E.
Lipschitz stability of a non-standard problem for the non-stationary transport equation via a Carleman estimate.
Inverse Problems. 2006;22:881--890.

\bibitem{Klibanov-Yamamoto07}
Klibanov M V, Yamamoto M.
Exact controllability for the time dependent transport equation.
SIAM J Control Optim. 2007;46:2071--2195.

\bibitem{Lin-Nakamura16}
Lin C-L, Nakamura G.
Unique continuation property for anomalous slow diffusion equation.
Comm Partial Diff Eq. 2016;41:749--758.

\bibitem{Liu-Yamamoto-Yan15}
Liu J J, Yamamoto M, Yan L.
On the uniqueness and reconstruction for an inverse problem of the fractional diffusion process.
Appl Num Math. 2015;87:1--19.

\bibitem{Machida17}
Machida M.
The time-fractional radiative transport equation -- Continuous-time random walk, diffusion approximation, and Legendre-polynomial expansion.
J Math Phys. 2017;58:013301.

\bibitem{Machida-Yamamoto14}
Machida M, Yamamoto M.
Global Lipschitz stability in determining coefficients of the radiative transport equation.
Inverse Problems. 2014;30:035010.

\bibitem{Metzler-Klafter00}
Metzler R, Klafter J.
The random walk's guide to anomalous diffusion: a fractional dynamics approach.
Phys Rep. 2000;339:1--77.

\bibitem{Prilepko-Ivankov84}
Prilepko A I, Ivankov A L.
Inverse problems for the time-dependent transport equation. 
Sovviet Math Dokl. 1984;29:559--564.

\bibitem{Sokolov-etal00}
Sokolov I M, Klafter J, Blumen A.
Fractional kinetics. 
Phys Today. 2002;55:48--54.

\bibitem{Williams92}
Williams M M R.
Stochastic problems in the transport of radioactive nuclides in fractured rock.
Nucl Sci Eng. 1992;112:215--230.

\bibitem{Williams93}
Williams M M R.
Radionuclide transport in fractured rock a new model: application and discussion.
Ann Nucl Energy. 1993;20:279--297.

\bibitem{Xu-Cheng-Yamamoto11}
Xu X, Cheng J, Yamamoto M.
Carleman estimate for a fractional diffusion equation with half order and application.
Applicable Analysis. 2011;90:1355--1371.

\bibitem{Yamakawa12}
Yamakawa M.
A study on the dependence of CTRW parameters on particle sizes in porus media (Japanese) [master thesis].
Tsukuba (Japan): Department of Risk Engineering, University of Tsukuba; 2012.

\bibitem{Yamamoto09}
Yamamoto M.
Carleman estimates for parabolic equations and applications.
Inverse Problems. 2009;25:123013.

\bibitem{Yamamoto-Zhang12}
Yamamoto M, Zhang Y.
Conditional stability in determining a zeroth-order coefficient in a 
half-order fractional diffusion equation by a Carleman estimate.
Inverse Problems. 2012;28:105010.

\bibitem{Yuan-Yamamoto09}
Yuan G and Yamamoto M.
Lipschitz stability in the determination of the principal part of a parabolic equation.
ESAIM: Control Optim. Calc. Var. 2009;15:525--554.

\end{thebibliography}
\end{document}